\documentclass[10pt,reqno]{amsart}
\usepackage{latexsym}
\usepackage{graphicx}
\usepackage{fancyhdr,amssymb}

\usepackage{hyperref} 

\theoremstyle{plain}
\numberwithin{equation}{section}
\newtheorem{thm}{Theorem}[section]
\newtheorem{theorem}[thm]{Theorem}

\newtheorem{corollary}[thm]{Corollary} 

\theoremstyle{definition}

\newtheorem{conjecture}[thm]{Conjecture}  

\begin{document}
\fancyhead{}
\renewcommand{\headrulewidth}{0pt}
\fancyfoot{}
\fancyfoot[LE,RO]{\medskip \thepage}
\fancyfoot[LO]{\medskip MISSOURI J.~OF MATH.~SCI., SPRING 2022}
\fancyfoot[RE]{\medskip MISSOURI J.~OF MATH.~SCI., VOL.~34, NO.~1}

\setcounter{page}{1}

\title{Distribution of Square-Prime Numbers}
\author{Raghavendra N Bhat}
\address{\newline Raghavendra N Bhat\newline 
University of Illinois, Urbana Champaign \newline
Department of Mathematics\newline
1409 West Green Street\newline
Urbana, IL 61801
}
\email{rnbhat2@illinois.edu (Corresponding Author)}


\begin{abstract}


For $a \neq 1$ and $p$ prime, we define numbers of the form $pa^2$ to be Square-Prime (SP) Numbers. 
For example, 75 = 3 $\cdot$ 25; 108 = 3 $\cdot$ 36; 45 = 5 $\cdot$ 9. 
These numbers are listed in the OEIS as A228056. 
We study the properties of these numbers, their distribution/density and also develop a few claims on their distribution/density. 
We rely on computer programs to verify some conjectures up to large numbers.
\end{abstract}
\maketitle


\section{Introduction}


SP Numbers can be generated in linear time allowing us to study some interesting properties and generalize a few conjectures and claims that would be more difficult to prove for the set of prime numbers. For example, there are many unsolved problems along the lines of Goldbach and twin prime conjectures. The intention of this paper is to study SP numbers, their density with respect to primes and to also study their distribution on the number line.


\section{Terminologies, Axioms and Trivial Results}


\begin{enumerate}
\item There are infinitely many SP Numbers. This is trivially proved owing to the existence of infinite primes and infinite squares.
\item  There can be consecutive natural numbers who are both SP. An example is 27 and 28 (27 is $3 \cdot 9$ and 28 is $7 \cdot 4$) 
\item The product of two SP Numbers is always non-SP.
\item An SP gap is defined as the difference between two consecutive SP Numbers.
\item Consecutive natural numbers that are both SP are called SP Twins. These SP pairs have gap 1.
\end{enumerate}

\noindent\rule{0.84in}{0.4pt} \par
\medskip
\indent\indent {\fontsize{8pt}{9pt} \selectfont DOI: 10.35834/YYYY/VVNNPPP \par}
\indent\indent {\fontsize{8pt}{9pt} \selectfont MSC2020: 11N25, 11P99 \par}
\indent\indent {\fontsize{8pt}{9pt} \selectfont Key words and phrases: generalization of primes, distribution of prime-like numbers \par}

\thispagestyle{fancy}

\vfil\eject
\fancyhead{}
\fancyhead[CO]{\hfill DISTRIBUTION OF SQUARE-PRIME NUMBERS}
\fancyhead[CE]{R.~N.~BHAT  \hfill}
\renewcommand{\headrulewidth}{0pt}


\section{First Few SP Numbers}


Here are the first 100 SP numbers:
\begin{align*}
&8,12,18,20,27,28,32,44,45,48,50,52,63,68,72,75,76,80,92,98,99,108,\\
&112,116,117,124,125,128,147,148,153,162,164,171,172,175,176,180,\\
&188,192,200,207,208,212,236,242,243,244,245,252,261,268,272,275,\\
&279,284,288,292,300,304,316,320,325,332,333,338,343,356,363,368,\\
&369,387,388,392,396,404,405,412,423,425,428,432,436,448,450,452,\\
&464,468,475,477,496,500,507,508,512,524,531,539,548,549.
\end{align*}


\section{Theorems on SP Numbers}


\begin{theorem}
For all natural numbers $n$ large enough, the number of SP Numbers smaller than $n$ is asymptotic to $(\zeta(2)-1)\frac{n}{\log n}$.
\end{theorem}

\begin{proof}
From the prime number theorem \cite[pp.~111--115]{Davenport}, we know that the number of prime numbers $\leq$ $n$ equals $\frac{n}{\log n} + O(\frac{n}{\log^2 n})$.
Let $SP(n)$ denote the number of SP numbers smaller than $n$. $SP(n)$ is equal to the number of pairs $(a,p)$ where $a^2p$ $\leq$ $n$ where $p$ is a prime number and $a$ is a natural number $\geq$ 2. $SP(n)$ is
$$\sum_{a=2}^{\sqrt{n/2}}\pi\left(\frac{n}{a^2}\right).$$ Here, $\pi$ is the prime counting function. For arbitrary $A$ smaller than $n$, the head term is
$$\sum_{a=2}^{A}\pi\left(\frac{n}{a^2}\right).$$
The tail is
$$\sum_{a=A}^{\sqrt{n/2}}\pi\left(\frac{n}{a^2}\right).$$
Set $A$ to be $\log^2 n$. Thus, the head term is $$\sum_{a=2}^{\log^2 n}\pi\left(\frac{n}{a^2}\right).$$
We simplify this using PNT as follows:

$$n\sum_{a=2}^{\log^2 n}\frac{1}{a^2(\log n-\log (a^2))}$$
$$=\frac{n}{\log n}\sum_{a=2}^{\log^2 n}\frac{1}{a^2(1-\frac{2\log a}{\log n})}.$$
Using the properties of geometric series, this simplifies to
$$\frac{n}{\log n}\sum_{a=2}^{\log^2 n}\frac{1}{a^2}\left(1+O\left(\frac{\log a}{\log n}\right)\right)$$

$$=\frac{n}{\log n} \sum_{a=2}^{\log^2 n}\frac{1}{a^2} + O\left(\frac{n}{\log^2 n}\right).$$ \newline
Simplifying the main term, yields
$$\frac{n}{\log n}\sum_{a=2}^{\infty}\frac{1}{a^2} - \frac{n}{\log n}\sum_{a=\log^2n}^{\infty}\frac{1}{a^2}$$
$$=\frac{n}{\log n}\sum_{a=2}^{\infty}\frac{1}{a^2} + O\left(\frac{n}{\log^2 n}\right)$$
$$=(\zeta(2) - 1)\frac{n}{\log n} + O\left(\frac{n}{\log^2 n}\right).$$
Hence we have the head of the original asymptotic to be
$$=(\frac{\pi^2}{6} - 1)\frac{n}{\log n} + O\left(\frac{n}{\log^2 n}\right).$$
Simplifying the tail yields
$$\sum_{a=\log^2n}^{\sqrt{n/2}}\pi(\frac{n}{a^2}) \leq n\sum_{a=\log^2n}^{\infty}\frac{1}{a^2(\log n-\log (a^2))}=O\left(\frac{n}{\log^2 n}\right).$$
Finally, we have$$SP(n)=(\zeta(2) - 1)\frac{n}{\log n} + O\left(\frac{n}{\log^2 n}\right).$$ This completes the proof.
\end{proof}

\begin{corollary}
For a sufficiently large $n$, the number of prime numbers smaller than $n$ is more than the number of SP numbers smaller than $n$.
\end{corollary}

\begin{proof}
Since $\frac{\pi^2}{6} - 1$ is less than 1, we have $SP(n)$ $\leq$ $\pi(n)$.
\end{proof}

\begin{theorem}
There are infinitely many pairs of SP numbers for any existing SP gap.\\
\end{theorem}

\begin{proof}
We use Gauss's result on the general Pell equation \cite[Chapter 8, Theorem 2 on p.~57]{Mordell} that states that $x^2 - dy^2 = m$ has infinitely many solutions, if it has one solution.
Thus, if a gap $g$ occurs, we have $P_1, P_2, a, b$ such that 
$$P_1a^2 - P_2b^2 = g.$$
Multiplying by $P_1$ yields
\[
P_1a^2 - (P_1P_2)b^2 = P_1g.
\]
We now have a general Pell equation
\begin{equation} \label{E:Pellform}
x^2 - (P_1P_2)y^2 = P_1g 
\end{equation}
that has a solution. Thus, it has infinitely many solutions. In each solution, $P_1$ divides $(P_1P_2)y^2$ and $P_1g$. Hence, it also divides $x$. Let $x=P_1k$. 
Thus, dividing every solution to \eqref{E:Pellform} by $P_1$ yields $$P_1k^2 - P_2y^2 = g.$$ 
Since $P_1k^2$ and $P_2y^2$ are SP Numbers which have gap $g$, the proof is complete.
\end{proof}

As an immediate corollary, we have the existence of infinite SP twins, our existing solution being 27 and 28.


\section{Conjectures on SP Numbers}


These conjectures are all verified computationally beyond numbers in excess of $10^9$. We look at these claims briefly to get some ideas for future research.

\begin{conjecture}
Every large natural number is a sum of two SP Numbers. \end{conjecture}The Goldbach conjecture is probably the most famous conjecture about primes. It claims that every even number greater than or equal to 4 can be written as a sum of two primes. However, for SP Numbers, as there are no even-odd parity restrictions, we claim that every sufficiently large natural number, whether even or odd, can be written as a sum of two SP Numbers. This has been verified beyond $10^9$ and is valid for all numbers starting $3931 (27 + 3904)$. E.g., 4041=116 + 3925, 10216=12 + 10204. \cite{Helfgott}

\begin{conjecture}
Every SP Number greater than 27 can be written as a sum of two SP Numbers in at least one way.\end{conjecture} This follows from conjecture 1 but takes effect with a lower starting point (although specifically restricted to only SPs and not all naturals). E.g., 153 = 28 + 125.

\begin{conjecture}
There exists at least one SP number between consecutive square numbers greater than 500.\end{conjecture} This is already conjectured for primes \cite{Bazzanella}. Computational verification shows that it works for SP Numbers as well. E.g., Between 625 and 676 we have 637, which is 49 $\cdot$ 13.


\section {On the Last Digit of SP Numbers}


Using the fact that primes are asymptotically equidistributed mod 10 (i.e. we have approximately the same number of primes ending in 1,3,7 and 9), we can come up with similar asymptotic estimates for SP numbers ending in 1,3,7 and 9. For example, if we wish to estimate the count of SP numbers ending in 1 (i.e. congruent to 1 mod 10), we wish to look at pairs $(p,a)$ such that $pa^2$ $\equiv$ 1 mod 10. $p$ can be $\equiv$ 1, 3, 7 or 9 mod 10. However, if $pa^2$ has to end in 1, $p$ cannot end in 3 or 7, because square numbers cannot end in 7 or 3 respectively. Thus we have two cases, $p$ ending in 1 or 9. When $p \equiv 1$ mod 10, we require $a^2$ to end in 1 (so that $pa^2$ ends in 1 too). Thus, $a$ can be $\equiv$ 1 or 9 mod 10.
Thus, the number of pairs $(p,a)$ such that $pa^2 \equiv$ 1 mod 10 with $p$ ending in 1 is
$$\frac{1}{4}\sum_{k=1}^{(n-1)/10}\pi\left(\frac{n}{(10k + 1)^2}\right) + \pi\left(\frac{n}{(10k + 9)^2}\right).$$
Using the fact that prime numbers are asymptotically distributed mod 10, this can be approximated by:
$$\frac{1}{4}\frac{n}{\log n}\sum_{k=1}^{\infty}\frac{1}{(10k + 1)^2} + \frac{1}{(10k + 9)^2}$$
$$=\frac{1}{400}\frac{n}{\log n}\sum_{k=1}^{\infty}\frac{1}{(k + 1/10)^2} + \frac{1}{(k + 9/10)^2}$$
$$=\frac{1}{400}\frac{n}{\log n}(\zeta(2,1/10) + \zeta(2,9/10) -2).$$
Here we use the Hurwitz zeta function which denotes $\sum_{k=0}^{\infty}\frac{1}{(k+c)^s}$ as $\zeta(s,c).$
Similarly, if $p$ ends in 9, we require $a^2$ to end in 9. Thus, $a$ can be $\equiv$ 3 or 7 mod 10.
Therefore, the asymptotic estimate for SP numbers ending in 1 with $p$ ending in 9 is the following:
$$\frac{1}{400}\frac{n}{\log n}(\zeta(2,3/10) + \zeta(2,7/10) -2).$$
Combining the above results, we have Theorem \ref{T:lastdigit}.

\begin{theorem} \label{T:lastdigit}
The asymptotic estimate of SP numbers ending in 1 is
$$\frac{1}{400}\frac{n}{\log n}(\zeta(2,1/10) + \zeta(2,9/10) + \zeta(2,3/10) + \zeta(2,7/10) - 4).$$
\end{theorem}

Using the process employed above, we find the same asymptotic estimate for SP numbers ending in 3, 7, and 9.

\medskip
\noindent
{\textbf{Acknowledgements.}}
I would like to acknowledge the professors, at the University of Illinois for their valuable advice and support. I also appreciate the comments from the editor, especially the guidance for proving the results about the last digits of SP numbers.



\begin{thebibliography}{99}


\bibitem{OEIS}
The On-Line Encyclopedia of Integer Sequences, published electronically at
\url{https://oeis.org.} Sequence A228056. \url{https://oeis.org/A228057}

\bibitem{Davenport}
Davenport, H., \emph{Multiplicative Number Theory}, Third Edition.
Springer, 2000.

\bibitem{Mordell}
Mordell, L. J., \emph{Diophantine Equations}, Academic Press, 1969.

\bibitem{Helfgott}
Helfgott, H., \emph{The ternary Goldbach conjecture is true}, arXiv:1312.7748

\bibitem{Bazzanella}
Bazzanella, D., \emph{Some conditional results on primes between consecutive squares}, Functiones et Approximatio Commentarii Mathematici, Funct. Approx. Comment. Math. 45(2), 255-263, (December 2011)

\end{thebibliography}
\end{document}